\documentclass[11pt,a4paper,reqno]{amsart}
\usepackage{hyperref}
\usepackage{amssymb, amsmath, amsthm}
\usepackage{bbm}

\numberwithin{equation}{section}

\def\R{\mathbb{R}}
\def\N{\mathbb{Z}_{\ge 0}}
\def\C{\mathbb{C}}
\def\Z{\mathbb{Z}}
                       
\newcommand{\pr}[1]{\mathrm{P}_{#1}}

\theoremstyle{plain}
\newtheorem{thm}{Theorem}[section]
\newtheorem{prop}[thm]{Proposition}

\newtheorem{lem}[thm]{Lemma}

\newtheorem{claim}[thm]{Claim}

\newtheorem*{conj*}{Conjecture}
\newtheorem*{openproblem*}{Open Problem}

\theoremstyle{definition}
\newtheorem{defi}[thm]{Definition}

\theoremstyle{remark}

\begin{document}

\title[A polynomial Roth theorem on the real line]{A polynomial Roth theorem on the real line}

\author[P. Durcik]{Polona Durcik}
\address{Polona Durcik, University of Bonn, Endenicher Allee 60, 53115 Bonn, Germany}
\email{durcik@math.uni-bonn.de}

\author[S. Guo]{Shaoming Guo}
\address{Shaoming Guo, Indiana University Bloomington, 831 E Third St, Bloomington, IN 47405, USA}
\address{Current address: Department of Mathematics, the Chinese University of Hong Kong, Hong Kong, China}
\email{shaomingguo2018@gmail.com}

\author[J. Roos]{Joris Roos}
\address{Joris Roos, University of Bonn, Endenicher Allee 60, 53115 Bonn, Germany}
\email{jroos@math.uni-bonn.de}

\date{\today}

\begin{abstract}
For a polynomial $P$ of degree greater than one, we show the existence of patterns of the form $(x,x+t,x+P(t))$ with a gap estimate on $t$ in positive density subsets of the reals. This is an extension of an earlier result of Bourgain. Our proof is a combination of Bourgain's approach and more recent methods that were originally developed for the study of the bilinear Hilbert transform along curves.
\end{abstract}

\maketitle

\section{Introduction}

Let $P: \R\to \R$ be a polynomial. We will let $\|P\|$ denote the $\ell^1$ sum of the coefficients of $P$. The main result of this paper is the following.
\begin{thm} \label{main-result}
Let $M\ge 1$ and $N\ge 1$ be real numbers. Let $\varepsilon>0$ be given and $S$ be a measurable subset of $[0, N]$   with $|S|\ge \varepsilon N$. Let $P: \R\to \R$ be a monic polynomial of degree $d>1$ without constant term that satisfies $ \|P\|\le M$. Then there exists $\delta(\varepsilon, M, d)>0$ such that we can find
\begin{equation}\label{pattern-real}
x, x+t, x+P(t)\in S
\end{equation}
with $t>\delta(\varepsilon, M, d) N^{1/d}$ and $\delta=\delta(\varepsilon,M,d)$ satisfies the estimate $(\log\log\delta^{-1})^{-\frac16} \gtrsim_{M,d} \varepsilon$.
\end{thm}

When $P(t)=2t$, then \eqref{pattern-real} is a consequence of the classical Roth theorem \cite{Roth}. See also \cite{Bou0} for an alternative proof and extensions to results of Szemer\'edi type. In the special case $P(t)=t^{d}$ with $d\neq 1$, Theorem \ref{main-result} is due to Bourgain \cite{Bou}. We extend his result to general polynomials.

A standard argument based on Lebesgue's density theorem shows that our theorem would hold trivially if we only asked for $t$ to be positive. Thus, the main point of our result is the gap estimate giving a lower bound on $t$. A related work that studies the existence of certain polynomial patterns in fractal subsets of $\R^n$ is \cite{HLP15}. The Lebesgue density argument does not apply in that case, so the goal in that paper is establishing existence of certain patterns without a gap estimate.

After concluding our work on this result we were made aware that a less quantitative variant of our theorem (without the estimate on $\delta$) can also be deduced from the main result in \cite{BL96}.

For the integers, the problem of searching for polynomial patterns in various sets, for instance the primes, has been studied intensively. We refer to \cite{TZ}, \cite{TZ2} and the references contained therein.

Our proof of Theorem \ref{main-result} is closely related to the study of the bilinear Hilbert transform along polynomial curves. Define 
\begin{equation}\label{bht-polynomial}
\mathcal{H}_P(f, g)(x):=p.v.\int_{\R} f(x-t)g(x-P(t))\frac{dt}{t}.
\end{equation}
If $P(t)=2t$, this is the classical bilinear Hilbert transform, which is the subject of Lacey and Thiele's  breakthrough papers \cite{LT1}, \cite{LT2} and has since been studied extensively. 
For certain nonlinear $P$ the operator in \eqref{bht-polynomial} has recently been studied in \cite{Li}, \cite{Lie}, \cite{LiXiao}, \cite{GuoXiao}, \cite{Lie2}. We invite the reader to consult these papers to learn about the development of this subject.

Another closely related object is the Hilbert transform along the curve $(t, P(t))_{t\in \R}$. For a function $f: \R^2\to \R$, we let 
\begin{equation}\label{hilbert-transform-curve}
H_P(f)(x, y)=p.v.\int_{\R} f(x-t, y-P(t))\frac{dt}{t}.
\end{equation}
In fact, the operators $\mathcal{H}_P$ and $H_P$ share the same Fourier multiplier. We refer to \cite{GHLR} and the references contained therein for historical background on the study of the operator \eqref{hilbert-transform-curve}.

The principle of using estimates for multilinear singular integrals to study patterns in subsets of the Euclidean space has also been  used elsewhere in the recent literature (see  \cite{CMP}, \cite{DKR}).\\

We now turn to describing the structure of the proof of Theorem \ref{main-result}. From now on $P$ will be a fixed monic polynomial of degree $d>1$ satisfying $\|P\|\le M$ and lacking a constant term.
Let $f$ be a function on $\R$ such that $0\le f\le 1$ and $\int_0^N f\ge \varepsilon N$. Assume that we could prove
\begin{equation*}\int_0^N \int_0^{N^{1/d}} f(x) f(x+t) f(x+P(t))dt dx > \delta(\varepsilon, M, d)N^{1+\frac{1}{d}}.
\end{equation*}
Then Theorem \ref{main-result} follows immediately by setting $f=
\mathbbm{1}_S$. 
While possibly changing $\delta$ by multiplication with a constant depending only on $d$, we may assume without loss of generality that $N=2^{jd}$ for some $j\in \N$. Changing variables $x\to Nx$ and $t\to N^{1/d}t$ and replacing $f(x)$ by $f(N^{-1}x)$, we see that it suffices to show 
\begin{align}\label{toshow}
\int_0^1  \int_0^{1} f(x) f(x+ 2^{-(d-1)j} t) f(x+2^{-dj}P(2^j t))dt dx > \delta(\varepsilon, M, d)
\end{align} 
for all functions $f$ with $0\le f\le 1$, $\int_0^1 f \ge \varepsilon$.

In the case $P(t)=t^d$ with $d\neq 1$, Bourgain \cite{Bou} proved \eqref{toshow} for all $j\in \N$. Note that in this case $P$ is scaling-invariant in the sense that $2^{-dj}P(2^j t)=P(t)$. For a general polynomial, we do not know how to prove \eqref{toshow} for all $j\in \N$. However we can prove \eqref{toshow} for sufficiently many $j$.
\begin{defi}Let $\Gamma_d$ be a constant depending only on the degree $d$ that is to be determined later.  A set $\mathcal{E}=\{j_1<j_2<\cdots\}\subset \N$ is called \emph{admissible} if
\begin{equation*}
j_1\le\Gamma_d\text{ and }j_{i+1}-j_i\le \Gamma_d \text{ for every } i\ge 1. 
\end{equation*}
\end{defi}
To prove Theorem \ref{main-result}, it suffices to show \eqref{toshow} for all $j$ contained in an admissible set $\mathcal{E}$. 
\begin{prop}\label{main-prop}There exists an admissible set $\mathcal{E}\subset\N$ with $0\in\mathcal{E}$ such that for every $\varepsilon>0$, we can find $\delta>0$ with $(\log\log\delta^{-1})^{-\frac16}\gtrsim_{M,d} \varepsilon$ such that for every $0\le f\le 1$ with $\int_0^1 f\ge \varepsilon$ we have 
\begin{equation}\label{0223e1.3}
\int_0^1  \int_0^{1} f(x) f(x+2^{-(d-1)j}t) f(x+2^{-dj}P(2^j t))dt dx > \delta,
\end{equation}
for every $j\in\mathcal{E}$.
\end{prop}

Let $\tau$ be a non-negative smooth bump function supported in $[1/2, 2]$ with integral $1$. For $\ell\in\N$ we denote $\tau_\ell(x)=2^{\ell}\tau(2^{\ell} x)$. The key lemma in the proof of Proposition \ref{main-prop} is the following.
\begin{lem} \label{main-lemma}
There exists $\gamma>0$ and admissible sets $\mathcal{E},\Lambda$, $0\in\mathcal{E}$, such that for every $(j,\ell)\in\mathcal{E}\times \Lambda$ and every test function $g$ with $\mathrm{supp}(\widehat{g})$ contained in $[2^m, 2^{m+1}]$, $m\geq 0$, we have 
\begin{equation}\label{main-estimate-3}
\Big\| \int_\R f(x+2^{-(d-1)j}t)g(x+2^{-dj}P(2^j t))\tau_{\ell}(t)dt\Big\|_{L^1([0,1])}  \le C_\ell 2^{-\gamma m}\|f\|_{2} \|g\|_{2},
\end{equation}
where $C_\ell\le 2^{\gamma_{d, M} \ell}$ for some $\gamma_{d, M}>0$ depending only on $d$ and $M$.
\end{lem}
In the case that $P$ is a monomial, Bourgain \cite{Bou} proved Lemma \ref{main-lemma} for all $j,\ell\in\N$. In Section \ref{sec:reduction} we show how the lemma implies Proposition \ref{main-prop}. The admissible sets $\mathcal{E}$ and $\Lambda$ are constructed in Section \ref{sec:admissiblesets}. In Section \ref{sec:main} we prove Lemma \ref{main-lemma}.

Estimate \eqref{main-estimate-3} is related to certain estimates for the bilinear Hilbert transforms along curves that appeared in \cite{Li}, \cite{Lie}, \cite{LiXiao}, \cite{GuoXiao}, \cite{Lie2}.
This enables us to adapt the approach used in these papers to prove Lemma \ref{main-lemma}. 
One difference to the congruent estimates for the bilinear Hilbert transform along curves is that \eqref{main-estimate-3} contains an extra scaling parameter.
Another difference is that we are allowing $P$ to have a linear term, while the methods described in the present literature for the bilinear Hilbert transform along curves cannot handle linear terms.
In particular, the problem of bounding $\mathcal{H}_P$ for $P$ being a polynomial of degree greater than one that includes a linear term is still open.

For the proof of Lemma \ref{main-lemma} we borrow a basic idea from the study of bilinear Hilbert transforms along curves: we treat a general polynomial as a perturbation of whatever monomial is dominating at each scale. However, at those scales where the linear term is dominating this turns out not to be enough. In that case we need to go one step further and see which of the remaining monomials is dominating the difference of the polynomial and its linear term.\\

{\bf Notation.} Throughout this paper, we will write 
$x\lesssim y$ to mean that there exists a constant $C$ depending only on fixed parameters depending on context (such as the degree $d$ of the polynomial $P$) such that $x\le C y$. We write $x\lesssim_p y$ to denote dependence of the implicit constant on the parameter $p$. Similarly, we define $x\approx y$ to mean that $x\lesssim y$ and $y\lesssim x$. $\mathbbm{1}_E$ will always denote the characteristic function of the set $E$.\\

{\bf Acknowledgements.} The authors thank Christoph Thiele for a helpful discussion of Bourgain's approach. They also thank Pavel Zorin-Kranich for pointing out the connection to Bergelson and Leibman's work. The second and third authors are indebted to Xiaochun Li, Victor Lie and Lechao Xiao for their generosity and numerous discussions on bilinear Hilbert transforms along curves. The first author is supported by the Hausdorff Center for Mathematics. The third author is supported by the Hausdorff Center for Mathematics and the German National Academic Foundation. This material is based upon work supported by the National Science Foundation under Grant No. DMS-1440140 while the authors were in residence at the Mathematical Sciences Research Institute in Berkeley, California, during the Spring semester of 2017.

\section{Reduction to the main lemma}\label{sec:reduction}
In this section we derive the   estimate \eqref{0223e1.3} from Lemma \ref{main-lemma}. The derivation is a straightforward adaptation of Bourgain's argument \cite{Bou} to our setting.
Let $\vartheta$ be  a non-negative even smooth  function supported on $[-2, 2]$, constant on $[-1,1]$, and monotone on $[1,2]$. We normalize it such that  $\widehat{\vartheta}(0)=1$ and denote $\vartheta_\ell(x):=2^{\ell}\vartheta(2^{\ell}x)$.

\begin{lem}[Bourgain \cite{Bou}]\label{lemma:bourgain}
For a non-negative function $f$ supported on $[0,1]$ and  $k,l\in \N$ we have 
\begin{equation*}
\int_0^1  f (f* \vartheta_{k}) (f* \vartheta_{\ell}) \ge c_0 \Big(\int_0^1 f\Big)^3
\end{equation*}
for some constant $c_0>0$ depending only on the choice of $\vartheta$.
\end{lem}
We include a proof of Lemma \ref{lemma:bourgain}, which was omitted in \cite{Bou}.
\begin{proof}
In this proof   all intervals are dyadic, that is, of the form $[2^km, 2^k(m+1)]$ for $k,m\in \Z$. 
For $k\in \Z$ we denote  by $E_k$ the dyadic martingale averages  
$$E_kf = \sum_{|I|=2^{-k}} \Big( \frac{1}{|I|}\int_I f \Big ) 1_I .$$
We claim that for any $0\leq k\leq \ell $ we have
\begin{align*}
\int_0^1  f(E_{k}f)(E_{\ell}f) \geq  \Big( \int_0^1  f \Big)^3  .
\end{align*}
Once this is shown, Lemma \ref{lemma:bourgain} follows by bounding the dyadic averages pointwise from above by the continuous averages $f*\vartheta_\ell$.

To see the claim we first observe that for any dyadic interval $J\subseteq [0,1]$ and any $k\in \Z$ with $2^{-k}\leq |J|$ we have by Cauchy-Schwarz
\begin{equation*}
\begin{split}
\Big(\frac{1}{|J|}  \int_J f \Big )^2& = \Big(\frac{1}{|J|} \sum_{I\subseteq J, \, |I|=2^{-k} }  \int_I f \Big )^2\\
&\leq \frac{1}{|J|} \sum_{I\subseteq J, \, |I|=2^{-k} }  \frac{1}{|I|} \Big(  \int_I f   \Big) ^2 =   \frac{1}{|J|} \int_J f\,(E_k f).
\end{split}
\end{equation*}
Combining this estimate with H\"older's inequality we obtain
\begin{align*}
\Big( \int_{0}^1 f \Big )^3 &  \leq \sum_{J\subseteq [0,1],\, |J|=2^{-\ell}}  \frac{1}{|J|^2} \Big ( \int_J f  \Big ) ^3 \\
&\leq \sum_{J\subseteq [0,1],\, |J|=2^{-\ell}} \Big ( \frac{1}{|J|} \int_J f \Big )\int_J f \, (E_kf)= \int_{0}^1 f  (E_kf) (E_\ell f),
\end{align*} 
which proves the claim.
\end{proof}

Now we are ready to deduce Proposition \ref{main-prop} from Lemmata \ref{main-lemma} and \ref{lemma:bourgain}.  
\begin{proof}[Proof of Proposition \ref{main-prop}] 
By localization in $x$ we may assume that    $ \textup{supp}(f)\subseteq [0,1]$. 
Denote
\begin{equation*}
I =\int_0^1\int_0^1  f(x)f(x+2^{-(d-1)j} t) f(x+2^{-dj}P(2^j t))dxdt.
\end{equation*}
For $\ell', \ell, \ell'' \in \Lambda$ with $1\leq \ell'\leq \ell \leq \ell'' $ we have 
\begin{align*}
 2^{\ell}I& \ge \int_0^1\int_0^1  f(x) f(x+2^{-(d-1)j}{t})f(x+2^{-dj}P(2^j t))\tau_{\ell}(t)dxdt\\ 
 &= I_1 + I_2 + I_3,
\end{align*}
 where
\begin{align*}
I_1 & =\int_0^1\int_0^1  f(x) f(x+2^{-(d-1)j}{t})(f*\vartheta_{\ell'})(x+2^{-dj}P(2^j t))\tau_{\ell}(t)dxdt,\\ 
I_2 & = \int_0^1\int_0^1  f(x) f(x+2^{-(d-1)j}{t})(f*\vartheta_{\ell''}-f*\vartheta_{\ell'})(x+2^{-dj}P(2^j t))\tau_{\ell}(t)dxdt,\\
I_3 & = \int_0^1\int_0^1  f(x)f(x+2^{-(d-1)j}{t})(f-f*\vartheta_{\ell''})(x+2^{-dj}P(2^j t)) \tau_{\ell}(t)dxdt. 
\end{align*}

We analyze each of the  terms separately.
Splitting $f-f*\vartheta_{\ell''}$ into Littlewood-Paley pieces and applying Lemma \ref{main-lemma}, it follows that for some $\sigma>0$ we have
\begin{align*}
 |I_3| \leq  2^{\gamma_{d, M}\ell-\sigma \ell''} \|f\|_{L^2(\R)}^2 \leq 2^{-100}c_0 \varepsilon^3,
\end{align*}  
where the last inequality holds provided that $\ell''$ is taken large enough with respect to $\ell$. Here $c_0$ is the constant from Lemma \ref{lemma:bourgain}.

To estimate $I_2$ we apply the Cauchy-Schwarz inequality in $x$, which yields
\begin{align*}
|I_2| &\le \int_0^1 \|f(x)f(x+2^{-(d-1)j}{t})\|_{L_x^2} \|(f*\vartheta_{\ell''}-f*\vartheta_{\ell'})(x+2^{-dj}P(2^j t))\|_{L_x^2} \tau_{\ell}(t) dt\\
&\leq \|f*\vartheta_{\ell''}-f*\vartheta_{\ell'}\|_2. 
\end{align*}
Passing to the last line we bounded the $L^\infty$ norm of $f$ and the $L^1$ norm of $\tau_{\ell}$ by one.

To estimate $I_1$ we compare it with 
\begin{align*}
I_4&=\int_0^1\int_0^1  f(x)f(x+2^{-(d-1)j}{t})(f*\vartheta_{\ell'})(x)    \tau_{\ell}(t)dx dt\\
&=\int_0^1 f(x) (f*\vartheta_{\ell'})  (x) (f* \tau_{\ell+(d-1)j})(x)dx.
\end{align*}
Consider the difference
\begin{equation*}
I_4-I_1= \int_0^1\int_0^1  f(x) f(x+2^{-(d-1)j}{t})\big( (f*\vartheta_{\ell'})  (x)-(f*\vartheta_{\ell'}) (x+2^{-dj}P(2^j t))\big ) \tau_{\ell}(t)dx dt.
\end{equation*}
By the mean value theorem we obtain  
\begin{equation*}
|(f*\vartheta_{\ell'})(x)-(f*\vartheta_{\ell'})(x+2^{-dj}P(2^j t))|\leq  2^{\ell'}\|f*(\vartheta')_{\ell'}  \|_\infty|2^{-dj}P(2^j t)| \leq dM  {2^{\ell'-\ell+1}},
\end{equation*}
whenever $t$ is in the support of $\tau_{\ell}$.
Choosing $\ell$ large enough with respect to $\ell'$ gives
\begin{equation*}
|I_4-I_1|\leq  2^{-100}c_0\varepsilon^3.
\end{equation*}
We return to analyzing the term $I_4$, which we write as
\begin{align}\label{termI41}
I_4 &= \Big( \int_0^1  f(x) (f*\vartheta_{\ell'})(x) \big ((f* \tau_{\ell+(d-1)j})(x)-(f* \vartheta_{\ell'+(d-1)j})(x) \big )dx\Big)\\ \label{termI42}
&+ \Big( \int_0^1  f(x) (f*\vartheta_{\ell'})  (x) (f* \vartheta_{\ell'+(d-1)j})(x)dx \Big ) 
\end{align}
By Lemma \ref{lemma:bourgain}, the term \eqref{termI42} 
is bounded from below by $c_0 \varepsilon^3$. 
For \eqref{termI41} we use   the triangle inequality and  Young's convolution inequality to estimate
\begin{align} \nonumber
& \|f*\tau_{\ell+(d-1)j}-f*\vartheta_{\ell'+(d-1)j}\|_2\\ \label{termI421}
& \leq  \|(f*\tau_{\ell+(d-1)j}*\vartheta_{\ell''+(d-1)j})-(f*\vartheta_{\ell'+(d-1)j}*\tau_{\ell+(d-1)j})\|_2\\ \label{termI422}
&   +    \|\tau_{\ell+(d-1)j}-(\tau_{\ell+(d-1)j}*\vartheta_{\ell''+(d-1)j})\|_1\\
\label{termI423}
& + \|\vartheta_{\ell'+(d-1)j}-(\vartheta_{\ell'+(d-1)j}*\tau_{\ell+(d-1)j})\|_1
  \end{align}
By another application of Young's convolution inequality in \eqref{termI421} and rescaling in \eqref{termI422} and \eqref{termI423}, we bound the last display by 
\begin{align*}
 \|(f*\vartheta_{\ell''+(d-1)j})-(f*\vartheta_{\ell'+(d-1)j})\|_2
  +   \|\tau_{\ell}-(\tau_{\ell}*\vartheta_{\ell''})\|_1 + \|\vartheta_{\ell'}-(\vartheta_{\ell'}*\tau_{\ell})\|_1.
\end{align*}
By the mean value theorem, the second and third term are bounded from above by    $2^{-100}c_0   \varepsilon^3$  provided $\ell''$ is chosen large enough with respect to $\ell$, and $\ell$ large enough with respect to $\ell'$. This in turn bounds \eqref{termI41}  from above by 
$$ \|f*\vartheta_{\ell''+(d-1)j}-f*\vartheta_{\ell'+(d-1)j}\|_2 + 2^{-99}c_0\varepsilon^3.$$
 
From   the estimates for the terms $I_1,I_2,I_3,I_4$ and $I_4-I_1$     we  obtain
\begin{align*}
c_0\varepsilon^3 \leq 2^{\ell}I +  \|f*\vartheta_{\ell'} -f*\vartheta_{\ell''}\|_2+\|f*\vartheta_{\ell'+(d-1)j}-f*\vartheta_{\ell''+(d-1)j}\|_2 + 2^{-90}c_0\varepsilon^3
\end{align*}
Therefore, we either have  
 $I> 2^{-\ell-10} c_0 \varepsilon^3$,
or 
\begin{equation*}
\|f*\vartheta_{\ell'} -f*\vartheta_{\ell''}\|_2+\|f*\vartheta_{\ell'+(d-1)j}-f*\vartheta_{\ell''+(d-1)j}\|_2> 2^{-10}c_0\varepsilon^3.
\end{equation*}
By the preceding discussion we can construct  a sequence $\{\ell_0<\ell_1<\cdots < \ell_k <\cdots\}\subseteq\Lambda$, which is independent of $f$ and $j$ and satisfies $\ell_{k+1} \le C \ell_k$ for some sufficiently large constant $C$ that depends on $M,\varepsilon,d$ such that for each $k$
either $$I> 2^{-\ell_{k+1}-10} c_0   \varepsilon^3$$ or 
\begin{align}\label{secondalt}
\|f*\vartheta_{\ell_k} -f*\vartheta_{{\ell}_{k+1}}\|_2+\|f*\vartheta_{{\ell_k}+(d-1)j}-f*\vartheta_{\ell_{k+1}+(d-1)j}\|_2> 2^{-10} c_0 \varepsilon^3.
\end{align}
 Observe that  for any $K\geq 0$ one has
\begin{align}\nonumber
&\sum_{k=0}^K\Big( \|f*\vartheta_{\ell_k} -f*\vartheta_{\ell_{k+1}}\|^2_2+\|f*\vartheta_{{\ell_k}+(d-1)j}-f*\vartheta_{\ell_{k+1}+(d-1)j}\|^2_2\Big )\\\label{orthogonality}
& \leq C_0 \|f\|^2_2\leq C_0
\end{align}
with $C_0$ independent of $K$ and $f$. Let us fix $K> C_02^{100}c_0^{-2}\varepsilon^{-6}$. 
If \eqref{secondalt} holds for all $0<k \le K$, then \eqref{orthogonality} yields $K\leq C_02^{100}c_0^{-2}\varepsilon^{-6}$, which is a contradiction. Thus, for some $0\le k \le K$ we necessarily have $I>2^{-\ell_{k+1}-10}  c_0 \varepsilon^3$. Together with $\ell_{k+1} \le C \ell_k$ and $\ell_0\leq \Gamma_d$
 this gives the lower estimate on $I$ claimed in Proposition \ref{main-prop}.
\end{proof}

\section{Construction of admissible sets}\label{sec:admissiblesets}

In this section we construct the admissible sets $\mathcal{E}$ and $\Lambda$. We write $P(t)=t^d+a_{d-1} t^{d-1}+\dots+ a_2 t^2+a_1 t$ and let $\Gamma_0$ be a large number depending only on $d$, say $\Gamma_0=2^{100 d!}$. 
 The precise value of $\Gamma_0$ is irrelevant. Define
\begin{equation*}
\mathcal{J}_{r}=\{k\in \Z: |a_{r} (2^{k})^{r}|> \Gamma_0|a_{r'} (2^{k})^{r'}| \text{ for every } r'\neq r\}
\end{equation*}
for $r=1,\dots,d$ and similarly,
\begin{equation*}
\mathcal{J}_{1,r}=\mathcal{J}_1\cap \{k\in \Z: |a_{r} (2^{k})^{r}|> \Gamma_0|a_{r'} (2^{k})^{r'}| \text{ for every } r'\not\in\{1,r\}\}
\end{equation*}
for $r=2,\dots,d$.
Roughly speaking, $\mathcal{J}_r$ can be understood as the set of dyadic scales $k$, where the $r$th power monomial dominates the behavior of the polynomial (and its derivatives).
We further denote
\begin{equation*}
\mathcal{J}_{\mathrm{good}}=\bigcup_{2\le r\le d} \mathcal{J}_{r}\cup \mathcal{J}_{1,r}\quad\text{and}\quad\mathcal{J}_{\mathrm{bad}}=\Z\setminus \mathcal{J}_{\mathrm{good}}.
\end{equation*}
Then the following variant of a lemma of Li and Xiao \cite{LiXiao} holds.
\begin{lem}\label{bad-set}
We have 
\begin{equation}\label{bad-small}
|\mathcal{J}_{\mathrm{bad}}| \le \Gamma_{d}.
\end{equation}
Here $\Gamma_d$ is a constant that depends only on $d$.
\end{lem}
\begin{proof}
This lemma is a slight variant of Lemma 2.1 in \cite{LiXiao}. We include the proof for the sake of completeness. The claim is that 
\begin{equation}\label{bad-small-3.5}
|\Z\setminus \cup_{1\le r\le d}\mathcal{J}_r|\le \Gamma_d.
\end{equation}
Estimate \eqref{bad-small} then follows from applying this estimate first to the polynomial $P$ and then to the polynomial $t\mapsto P(t)-a_1 t$.
To prove this estimate, we define $\mathcal{J}(r', r'')$ to be the collection of integers $k$ with
\begin{equation*}
|a_{r'}| 2^{r' k}\Gamma_0\ge |a_{r''}|2^{r'' k}\ge |a_{r'}| 2^{r' k}\Gamma_0^{-1}.
\end{equation*}
It is not difficult to see that $|\mathcal{J}(r', r'')|\le 4\Gamma_0$. Moreover, 
\begin{equation*}
(\Z\setminus \bigcup_{1\le r\le d}\mathcal{J}_r) \subset \bigcup_{1\le r'< r''\le d} \mathcal{J}_{r', r''}.
\end{equation*}
This proves \eqref{bad-small-3.5} with $\Gamma_d=4 d^2 \Gamma_0$.
\end{proof}
The good and bad sets for the rescaled polynomial $t\mapsto 2^{-dj} P(2^j t)$ are simply given by shifts of the good and bad sets for $P$. Accordingly, we define
\begin{equation*}
\mathcal{J}_r^{(j)} = \mathcal{J}_r - j
\end{equation*}
and similarly $\mathcal{J}^{(j)}_{1,r}, \mathcal{J}_\mathrm{good}^{(j)}$ and $\mathcal{J}_\mathrm{bad}^{(j)}$.

Now we construct the admissible sets $\mathcal{E}$ and $\Lambda$.
The set $\mathcal{E}$ will be chosen as a suitable subset  of
\begin{equation*}
\mathcal{E}_0=\{2\Gamma_d j: j\in\N\},
\end{equation*}
where $\Gamma_d$ is the constant from Lemma \ref{bad-set}.
We claim that the set 
\begin{equation*}
\Lambda_0=\N \setminus \bigcup_{j\in \mathcal{E}_0} -\mathcal{J}^{(j)}_{\mathrm{bad}}.
\end{equation*}
is admissible. Indeed, looking at residue classes modulo $2\Gamma_d$, we note that the cardinality of
\begin{equation*}
\Big(\bigcup_{j\in \mathcal{E}_0} -\mathcal{J}^{(j)}_{\mathrm{bad}} \;\mathrm{mod}\;2\Gamma_d\Big) =\Big( -\mathcal{J}_{\mathrm{bad}} \;\mathrm{mod}\; 2\Gamma_d\Big)
\end{equation*}
is at most $\Gamma_d$ by Lemma \ref{bad-set}. This proves the claim. By construction, we have that 
\begin{equation*}-\ell\in\mathcal{J}_{\mathrm{good}}^{(j)}
\end{equation*}
holds for every $j\in \mathcal{E}_0$ and $\ell\in \Lambda_0$. That is, the polynomial $t\mapsto 2^{-dj}P(2^j t)$ behaves like a monomial on the annulus $|t|\in[2^{-\ell}, 2^{-\ell+1}]$.

For $r=1, \dots, d$, we let $b_{r}$ be the unique integer such that\footnote{We assume without loss of generality that $a_r\not=0$.} 
\begin{equation}\label{eqn:brdef}
|a_{r}|\in [2^{b_{r}},2^{b_{r}+1}).
\end{equation}
For reasons that will become clear in Section \ref{sec:main} (see \eqref{0221e2.18} and the discussion below \eqref{eqn:quotientawayfrom1}) we require the condition
\begin{equation}\label{eqn:m0islarge}
|b_{r}+(r-1)(j-\ell)|\ge 10^{10 d} \text{ for every } 2\le r\le d.
\end{equation} 
to hold. 
Now we just pick
\begin{equation*} 0=j_0<\ell_0<j_1<\ell_1<\dots
 \end{equation*}
with $j_i\in\mathcal{E}_0$, $\ell_i\in\Lambda_0$ such that \eqref{eqn:m0islarge} holds (this is possible by looking at residue classes modulo $2\Gamma_d$ again). Then we set $\mathcal{E}=\{j_0,j_1,\dots\}$ and $\Lambda=\{\ell_0,\ell_1,\dots\}$.

\section{The main argument}\label{sec:main}

In this section we prove Lemma \ref{main-lemma}. Let $(j,\ell)\in\mathcal{E}\times\Lambda$, where $\mathcal{E}$ and $\Lambda$ are the sets constructed in the previous section. Then there exists $1\le d_0\le d$ such that $-\ell\in \mathcal{J}^{(j)}_{d_0}$. Thus we have by definition that
\begin{equation}\label{eqn:d0dominating}
|a_{d_0} (2^{j-\ell})^{d_0}|> \Gamma_0 |a_{r} (2^{j-\ell})^{r}|
\end{equation}
for all $r\not=d_0$. In other words, the $d_0$th order monomial dominates the absolute value of $P(x)$ at the scale $x\approx 2^{j-\ell}$. Note that the same automatically holds for all derivatives of the polynomial.
Recall the definition of $b_r$ from \eqref{eqn:brdef}. By \eqref{eqn:d0dominating} and since $a_d=1$, we have the lower bound
\begin{equation}\label{lower-bound-bd}
b_{d_0}\ge (j-\ell)(d-d_0).
\end{equation}
Notice that if we are in the case $m\le 100 d \ell$, then the desired estimate \eqref{main-estimate-3} will follow simply from Minkowski's and H\"older's inequalities, since the right hand side of \eqref{main-estimate-3} is allowed to depend on $\ell$ as indicated in Lemma \ref{main-lemma}.

In the rest of this section, we always assume that $m> 100 d\ell$.
The claim in Lemma \ref{main-lemma} about the dependence of the constant on $\ell$ and $M$ is easily seen by an inspection of the proof. In order to simplify notation, we will not make any further comments on this issue and merely indicate the dependence of inequalities on $\ell$ by writing $\lesssim_\ell$ and similarly for $M$.
By $\pr{k} f$ we denote the frequency projection defined by $$\widehat{\pr{k} f}(\xi)=\widehat{f}(\xi) \mathbbm{1}_{k}(\xi),$$
where $\mathbbm{1}_k(\xi)=\mathbbm{1}_{\{|\xi|\in 2^k[1,2)\}}$. Then the quantity we need to bound can be written as the $L^1([0,1])$ norm of
\begin{equation}\label{0210e2.7}
\sum_{k\in \Z}  \int_{\R} (\pr{k}f)(x+2^{-(d-1)j}{t}) g(x+2^{-dj}P(2^j t)) \tau_\ell(t)dt   .
\end{equation}
Passing to the Fourier side and looking at potential critical points of the phase we expect the main contribution to come from the case when
\begin{equation*}
k=b_{d_0} + (d_0-1)(j-\ell) + m.
\end{equation*}
This motivates us to write
\begin{equation}\label{0221e2.18}
m_0=b_{d_0}+(d_0-1)(j-\ell).
\end{equation}
Observe that \eqref{lower-bound-bd} implies the following lower bound on $m_0$:
\begin{equation}\label{eqn:lowerbdm0}
m_0\ge (d-1)(j-\ell).
\end{equation}
Now we write \eqref{0210e2.7} as 
\begin{equation*}
\sum_{k\in \Z}  \int_{\R} \pr{m+m_0+k}f (x+2^{-(d-1)j}{t}) g(x+2^{-dj}P(2^j t)) \tau_\ell(t)dt.
\end{equation*}
Let us first consider the case that $|k|$ is large, say greater than $10^d$. Due to the lack of critical points in the phase we do not expect a large contribution from this term. We dualize using $h\in L^\infty([0,1])$ and consider
\begin{equation*}
\sum_{|k|> 10^d}  \int_0^1 \int_{\R} \pr{m+m_0+k}f (x+2^{-(d-1)j}{t}) g(x+2^{-dj}P(2^j t)) \tau_\ell(t) h(x) dtdx.
\end{equation*}
By Fourier inversion, this can be written as (up to a universal constant)
\begin{equation*}\sum_{|k|>10^d} \iint_{\R^2} \widehat{f}(\xi) \widehat{g}(\eta) \mathbbm{1}_{m+m_0+k}(\xi) \widehat{h}(-\xi-\eta) \left(\int_{\R} e^{i\xi 2^{-(d-1)j-\ell}{t}+i\eta 2^{-dj}P(2^{j-\ell} t)} \tau(t)dt \right) d(\xi,\eta) .
\end{equation*}
Note that the $t$-derivative of the phase in the integral over $t$ is $\gtrsim 2^{-\ell d} 2^m$ (actually it is even larger for positive $k$, but we don't need to make use of any decay in $k$). This follows using \eqref{eqn:d0dominating} and \eqref{eqn:lowerbdm0} and that $|k|$ is large.
Therefore, integration by parts and an application of the Cauchy-Schwarz inequality to the integration in $(\xi,\eta)$ yields that the previous display is bounded as 
\begin{equation*}\begin{split}
\lesssim_{\ell} 2^{-m} & \left(\iint_{\R^2} |\widehat{g}(\eta)|^2 \sum_{|k|>10^d} \left|\widehat{f}(\xi)\mathbbm{1}_{m+m_0+k}(\xi)\right|^2 d(\xi,\eta)\right)^{1/2}\\
& \times \left(\iint_{\R^2} |\widehat{h}(-\xi-\eta)|^2 \mathbbm{1}_m(\eta) d(\xi,\eta)\right)^{1/2}.
\end{split}
\end{equation*}
Here we have used orthogonality of the functions $\mathbbm{1}_k$. The previous display is
\begin{equation*}\le 2^{-\frac{m}2} \|f\|_2 \|g\|_2 \|h\|_\infty,
\end{equation*}
where we have estimated $\|\widehat{h}\|_2=\|h\|_{L^2([0,1])}\le \|h\|_\infty$.

Thus it remains to treat the case when $|k|$ is small (bounded by a constant depending only on $d$). Without loss of generality we set $k=0$ to simplify notation. So for the remainder of this section, we will be concerned with the verification of the inequality
\begin{equation*}
\Big\| \int_{\R} f(x+2^{-(d-1)j}t) g(x+2^{-dj}P(2^j t)) \tau_\ell(t)dt\Big\|_{L^1([0, 1])} \lesssim_\ell 2^{-\gamma m} \|f\|_2 \|g\|_2
\end{equation*}
for some positive $\gamma>0$, $\widehat{f}$ supported in $2^{m+m_0} \leq |\xi| \leq  2^{m+m_0+1}$ and $\widehat{g}$ supported in $2^m\le |\xi|\le 2^{m+1}$. 
We first perform a few preliminary manipulations in order to streamline the argument. Changing variables $x\rightarrow 2^{-m-m_0}x$ we see that it suffices to show 
\begin{equation*}\begin{split}
& \Big\| \int_{\R} f(2^{-m-m_0}x+2^{-(d-1)j}t) g(2^{-m-m_0}x+2^{-dj}P(2^j t)) \tau_\ell(t)dt\Big\|_{L^1([0, 2^{m_0+m}])}\\
& \lesssim_\ell 2^{m+m_0} 2^{-\gamma m} \|f\|_2 \|g\|_2,
\end{split}
\end{equation*}
By a rescaling of $f$ and $g$ it suffices to show
\begin{equation*} \Big \|  \int_{\R} {f}(x +\lambda t)g(2^{-m_0}x+\lambda Q(t)) \tau(t)dt \Big \|_{L^1([0,2^{m_0+m}])}  \lesssim_\ell 2^{\frac{m_0}{2}} 2^{-\gamma m }  \|f\|_2\|g\|_2,
\end{equation*}
where $\widehat{f}$ and $\widehat{g}$ are supported in the annulus $1\le |\xi|\le 2$ and we have set
\begin{equation}\label{eqn:lambdaQdef}
\lambda = 2^{m+m_0-(d-1)j-\ell}\quad\text{and}\quad
Q(t)= 2^{-m_0+\ell-j} P(2^{j-\ell} t).
\end{equation}
By \eqref{eqn:lowerbdm0} we have
\begin{equation}\label{eqn:lambdaest}
\lambda \gtrsim_\ell 2^m. 
\end{equation}
Also note that $Q$ is well normalized in the sense that $\|Q\|_{C^d([1/2,2])}\approx 1$. Dualizing the $L^1([0,2^{m_0+m}])$ norm using $h\in L^\infty([0,2^{m_0+m}])$ it is enough to show that
\begin{equation}\label{dualized-bound-linf}
\Big| \int_\R \int_{\R} {f}(x +\lambda t)g(2^{-m_0}x+\lambda Q(t)) h(x) \tau(t) dtdx \Big| \lesssim_\ell 2^{\frac{m_0}2}2^{-\gamma m} \|f\|_2\|g\|_2 \|h\|_\infty
\end{equation}
By H\"older's inequality it suffices to verify that the trilinear estimate
\begin{equation}\label{dualized-bound}
\Big| \int_\R \int_{\R} {f}(x +\lambda t)g(2^{-m_0}x+\lambda Q(t)) h(x) \tau(t) dtdx \Big| \lesssim_\ell 2^{-\gamma m} 2^{-\frac12 m} \|f\|_2\|g\|_2 \|h\|_2
\end{equation}
holds for $f,g$ with Fourier support in $1\le |\xi|\le 2$.
Applying the Fourier inversion formula to $f$ and $g$, the integral on the left hand side of the previous display can be written as (up to a universal constant)
\begin{equation*}\iint_{\R^2} \widehat{f}(\xi) \widehat{g}(\eta) \widehat{h}(-\xi-2^{-m_0}\eta) \left( \int_{\R} e^{i \lambda (t\xi+Q(t)\eta)} \tau(t) dt \right) d(\xi,\eta) .
\end{equation*}
We denote the phase function of the integral in $t$ by
\begin{equation}\label{eqn:PhiDef} \Phi_{\xi,\eta} (t) =  t \xi + \eta Q(t).
\end{equation}

In the following we will always assume that the equation
\begin{equation*} \Phi'_{\xi,\eta}(t_c) = 0
\end{equation*}
has a unique solution $t_c=t_c(\xi,\eta)\in (1/2, 2)$. In the case of multiple solutions, each of them can be isolated by adding an appropriate cutoff function in $t$ (which we silently include into $\tau$) and each is then treated in the exact same way.
If on the other hand there is no such solution, we can integrate by parts and apply the Cauchy-Schwarz inequality similarly as above to obtain the desired bound (also see the discussion below \eqref{eqn:statphase}). Let us denote the dual phase function by
\begin{equation}\label{psi-function}
 \Psi(\xi,\eta) = \Phi_{\xi,\eta}(t_c) = t_c \xi + \eta Q(t_c). 
\end{equation}
Note that the existence of a critical point depends on the variables $\xi$ and $\eta$ (in the case $d_0>1$ it actually depends only on $\xi/\eta$). 
Before we proceed we state an incarnation of the stationary phase principle that will be invoked various times during the argument.
\begin{lem}\label{lemma:stationaryphase}
Suppose that there exists a unique $t_0\in(1/2,2)$ such that $\phi'(t_0)=0$. Assume $\phi''(t_0)\not=0$ and that $\tau$ is supported in $[1/2,2]$. Then
\begin{equation*} \int_\R e^{i\lambda\phi(t)} \tau(t) dt = \lambda^{-1/2} e^{i\lambda\phi(t_0)} (c |\phi''(t_0)|^{-\frac12} \tau(t_0) + R(\lambda)), 
\end{equation*}
with $R^{(r)}(\lambda)=O(\lambda^{-r-\frac12})$. Here, $c$ is a universal constant and the estimates of the remainder term $R$ depend only on finitely many derivatives of $\phi$ and $\tau$.
\end{lem}
The proof is standard and follows from \cite[Ch. VIII.1, Prop. 3]{Stein}, combined with appropriate integration by parts.

We now distinguish three cases. In the first two cases we assume that the dominating monomial is nonlinear, i.e. $d_0\ge 2$. In the third case we assume that the linear term is dominating, that is, $d_0=1$.

\subsection{Case I:  $d_0\ge 2$ and $|m_0|\le (1-\kappa)m$.} 
Here $\kappa$ is a small, positive universal constant that is to be determined later. In this case we follow the approach of \cite[Section 5]{Li}, and use the $TT^*$ method to obtain \eqref{dualized-bound}. 

As a consequence of Lemma \ref{lemma:stationaryphase} we have
\begin{equation}\label{eqn:statphase}
\int_{\R} e^{i \lambda\Phi_{\xi,\eta}(t)} \tau(t) dt = \lambda^{-1/2} e^{i\lambda \Psi(\xi,\eta)} a(\xi,\eta)  + R_{\xi,\eta}(\lambda),
\end{equation}
where $a(\xi, \eta)$ is a smooth and compactly supported function and the remainder term $R_{\xi,\eta}$ satisfies
\begin{equation*} |R_{\xi,\eta}(\lambda)|\lesssim \lambda^{-1}, 
\end{equation*}
where the implied constant depends only on $d$. In the case that $\xi,\eta$ are such that there exists no critical point (and therefore $\Psi(\xi,\eta)$ is not well-defined), we may set $a(\xi,\eta)=0$. As a consequence, that case is also handled by the remainder term. The reader should also keep this convention in mind for the remaining applications of stationary phase later in this section. We will not address this issue anymore from now on.

The function $a(\xi,\eta)$ depends on all our parameters, however does so in a harmless way. For instance, we have $|a(\xi,\eta)|\approx 1$ (and this information suffices for our purposes). The remainder term in \eqref{eqn:statphase} can be treated by an application of the Cauchy-Schwarz inequality in $(\xi,\eta)$. Indeed, we obtain
\begin{equation}\label{eqn:statphaseerr}
\left|\iint_{\R^2} \widehat{f}(\xi)\widehat{g}(\eta)\widehat{h}(-\xi-2^{-m_0}\eta) R_{\xi,\eta}(\lambda) d(\xi,\eta)\right|\lesssim_\ell 2^{-m} \|f\|_2 \|g\|_2 \|h\|_2.
\end{equation}
Here we used that $\lambda^{-1}\lesssim_\ell 2^{-m}$. Turning our attention to the main term, it now remains prove that
\begin{equation}\label{mainest:statphase}
\Big| \iint_{\R^2} \widehat{f}(\xi) \widehat{g}(\eta) e^{i \lambda \Psi(\xi,\eta)} a(\xi,\eta) \widehat{h}(-\xi-2^{-m_0}\eta) d(\xi,\eta) \Big| \lesssim_\ell 2^{-\gamma m} \|f\|_2 \|g\|_2 \|h\|_2.
\end{equation}
Changing variables from $\xi$ to $\xi-2^{-m_0}\eta$ and applying the Cauchy-Schwarz inequality to separate the function $h$, we see that it suffices to show
\begin{equation}\label{mainest:hoelder}
\Big\| \int_{\R} \widehat{f}(\xi-2^{-m_0} \eta) \widehat{g}(\eta) a(\xi-2^{-m_0} \eta, \eta) e^{i \lambda \Psi(\xi-2^{-m_0}\eta,\eta)} d\eta\Big\|_{L_\xi^2(\R)} \lesssim_\ell 2^{-\gamma m} \|f\|_{2} \|g\|_{2}.
\end{equation}
Expanding the square of the $L^2$ norm on the left hand-side gives   
\begin{equation*}
\begin{split}
 \int_\R & \Big( \int_\R \widehat{f}(\xi-2^{-m_0} \eta) \widehat{g}(\eta) a(\xi-2^{-m_0} \eta, \eta) e^{i \lambda \Psi(\xi-2^{-m_0}\eta,\eta)} d\eta \Big )\\
 & \times \Big( \int_\R \overline{\widehat{f}(\xi-2^{-m_0} \eta')\widehat{g}(\eta') a(\xi-2^{-m_0} \eta', \eta')} e^{-i \lambda \Psi(\xi-2^{-m_0}\eta',\eta')} d\eta' \Big ) d\xi.
\end{split}
\end{equation*}
The change of variables
\begin{equation}\label{eqn:TTstarchgofvar}
\eta'\to \eta-\alpha,\quad  \xi-2^{-m_0}\eta\to \xi
\end{equation}
transforms the left hand-side of \eqref{mainest:hoelder} into
\begin{equation*}\iiint_{\R^3}  F_\alpha(\xi) G_\alpha(\eta) \chi_\alpha(\xi,\eta) e^{i \lambda \Xi_\alpha(\xi,\eta)} d(\eta, \xi, \alpha),
\end{equation*}
where
\begin{align}\label{eqn:TTstarnotation}
F_\alpha(\xi) &= \widehat{f}(\xi)\overline{\widehat{f}(\xi+2^{-m_0}\alpha)}\\ \nonumber
G_\alpha(\eta) &= \widehat{g}(\eta)\overline{\widehat{g}(\eta-\alpha)},\\ \nonumber
\chi_\alpha(\xi,\eta) &= a(\xi, \eta)\overline{a(\xi+2^{-m_0}\alpha, \eta-\alpha)},\;\text{and}\\ \nonumber
\Xi_{\alpha}(\xi, \eta) &= \Psi(\xi,\eta)- \Psi(\xi+2^{-m_0}\alpha,\eta-\alpha).
\end{align}
We split the integration in $\alpha$ over the regions $|\alpha| \leq \alpha_0$ and $|\alpha| \geq \alpha_0$, where $\alpha_0>0$ is to be determined later. If $|\alpha| \leq \alpha_0$ we simply use the triangle inequality and Cauchy-Schwarz in $(\xi,\eta)$ to estimate
\begin{equation}\label{mainest:alphasmall}
\int_{|\alpha| \leq \alpha_0} \iint_{\R^2} |F_\alpha(\xi) G_\alpha(\eta) \chi_\alpha(\xi,\eta)| d(\eta,\xi) d\alpha \lesssim \alpha_0 \|f\|_{2}^2\|g\|_{2}^2.
\end{equation}
If $|\alpha| \geq \alpha_0$, the idea is to exploit the cancellation caused by the oscillation of the phase function $\Xi_\alpha$.
Our claim is that since $|m_0|$ is large (recall \eqref{eqn:m0islarge}) we have
\begin{equation}
\label{eqn:mixedderivest}
|\partial_\xi\partial_\eta \Xi_{\alpha}(\xi,\eta)|\gtrsim \alpha.
\end{equation}
By implicit differentiation we have
\begin{equation*} \partial_\eta t_c = -\frac{Q'(t_c)}{\eta Q''(t_c)} 
\end{equation*}
Moreover we compute
\begin{equation*} \partial_\xi \Psi(\xi,\eta) = t_c + \partial_\xi t_c\cdot \xi + \partial_\xi t_c \cdot \eta Q'(t_c) = t_c,
\end{equation*}
and as a consequence,
\begin{equation*} \partial_\eta \partial_\xi \Psi(\xi,\eta) = \partial_\eta t_c = -\frac{Q'(t_c)}{\eta Q''(t_c)}. 
\end{equation*}
Let us set $H(\xi,\eta)=\partial_\eta \partial_\xi \Psi(\xi,\eta)$.
By the mean value theorem we have
\begin{equation}\label{eqn:mixedderivmvt} H(\xi,\eta) - H(\xi+2^{-m_0}\alpha,\eta-\alpha) = \nabla H (\tilde{\xi},\tilde{\eta}) \cdot (-2^{-m_0}\alpha,\alpha),
\end{equation}
where $(\tilde{\xi},\tilde{\eta})$ is some convex combination of $(\xi,\eta)$ and $(\xi+2^{-m_0}\alpha,\eta-\alpha)$. We claim that
\begin{equation}\label{eqn:Hbound}
|\partial_\xi H(\xi,\eta)|\approx_d 1\,\,\text{and}\,\, |\partial_\eta H(\xi,\eta)|\approx_d 1.
\end{equation}
Indeed, we compute
\begin{equation*} 
\partial_\xi H(\xi,\eta) = \frac{1}{\eta^2 Q''} \left( 1 - \frac{Q''' Q'}{(Q'')^2}\right)
\,\,\text{and}\,\,
\partial_\eta H(\xi,\eta) = \frac{Q'}{\eta^2 Q''} \left( 2 - \frac{Q' Q'''}{(Q'')^2}\right).
\end{equation*}
It is clear that $|\nabla H(\xi,\eta)|\lesssim 1$. To obtain the lower bounds, we need to study the fraction
\begin{equation*} \frac{Q' Q'''}{(Q'')^2} = \frac{P'(2^{j-\ell}t_c)P'''(2^{j-\ell}t_c)}{P''(2^{j-\ell}t_c)^2}
\end{equation*}
Since $2^{j-\ell}t$ is in the range where the $d_0$th monomial dominates (i.e. \eqref{eqn:d0dominating} holds), we have that this is bounded by
\begin{equation*} \frac{(d_0+h)(d_0(d_0-1)(d_0-2)+h)}{(d_0(d_0-1)-h)^2} 
\end{equation*}
To see this recall that if a monomial dominates the absolute value of the polynomial at a certain scale, then it also dominates the absolute value of the derivatives of the polynomial at that scale. Here, $h>0$ can be made as small as we please by making $\Gamma_0$ larger, if necessary. 
By continuity, we can make $h$ small enough such that the above fraction is only $\frac1{100(d_0-1)}$ away from
\begin{equation*} \frac{d_0-2}{d_0-1} 
\end{equation*}
Thus,
\begin{equation}\label{eqn:quotientawayfrom1}
1-\frac{Q' Q'''}{(Q'')^2}\ge  1 - \frac{d_0-2}{d_0-1}-\frac1{100(d_0-1)} = \frac{99}{100}\cdot \frac{1}{d_0-1}\gtrsim_d 1.\end{equation}
We have therefore proven \eqref{eqn:Hbound}. Recall that we have chosen $j\in\mathcal{E}$ and $\ell\in\Lambda$. Thus, $|m_0|$ is large in the sense that \eqref{eqn:m0islarge} holds. Using this we obtain from \eqref{eqn:mixedderivmvt} that \eqref{eqn:mixedderivest} holds. This is because in the case that $m_0$ is large, the inner product on the right hand side of \eqref{eqn:mixedderivmvt} is dominated by the second component and if $-m_0$ is large, then it is dominated by the first component. Now we use the following well known fact.
\begin{lem}[H\"ormander]\label{lemma_hormander}
Let $a,\varphi$ be smooth functions in $\R^2$, $\varphi$ real-valued and $\lambda>1$. Also denote
\begin{equation*} T_\lambda f(x) = \int_\R e^{i\lambda\varphi(x,y)} a(x,y) f(y) dy.
\end{equation*}
Assume that $\partial_x\partial_y \varphi(x,y)\not= 0$ in the support of $a$. Then we have
\begin{equation*} |\langle T_\lambda f, g\rangle| \le C \lambda^{-1/2} \|f\|_2 \|g\|_2. 
\end{equation*}
\end{lem}
This is a dualized version of the $L^2$ endpoint of \cite[Thm. 1.1]{Hor72}. The proof is simply by $TT^*$ and stationary phase. Applying this result to our situation we conclude
\begin{equation}
\begin{split} \label{mainest:alphabig}
& \Big| \int_{|\alpha| \geq \alpha_0} \int_{\R}\int_{\R}  
F_\alpha(\xi) G_\alpha(\eta) e^{i\lambda \Xi_{\alpha}(\xi,\eta)} d\eta d\xi d\alpha \Big |\\
& \lesssim_\ell  (2^m \alpha_0)^{-\frac{1}{2}}  \int_{|\alpha|\ge \alpha_0} \|F_\alpha\|_2  \|G_\alpha\|_2 d\alpha \lesssim_\ell  (2^m \alpha_0)^{-\frac{1}{2}} 2^{\frac{m_0}2} \|f\|_{2}^2\|g\|_{2}^2,
\end{split}
\end{equation}
where the last inequality follows by Cauchy-Schwarz applied to the integration in $\alpha$.
Thus, if $|m_0|\le (1-\kappa)m$ for some fixed small absolute constant $\kappa>0$, then by letting $\alpha_0=2^{-\frac{\kappa m}{2}}$ we see that our desired estimate \eqref{mainest:statphase} holds with $\gamma=\kappa/8$.

\subsection{Case II: $d_0\ge 2$ and $|m_0|>(1-\kappa)m$.}
Here we apply a $\sigma$-uniformity argument in the spirit of \cite{Li} and \cite{LiXiao}. Alternatively, one can also follow the approach of \cite{Lie} which does not use the concept of $\sigma$--uniformity. Before we start, let us briefly review the basic setup of $\sigma$-uniformity. 
\begin{defi}[$\sigma$-uniformity]
Let $\sigma\in (0, 1)$, $\mathcal{I}\subset \R$ a bounded interval and $\mathcal{U}(\mathcal{I})$ a non-trivial subset of $L^2(\mathcal{I})$ such that $\sup_{u\in\mathcal{U}(\mathcal{I})} \|u\|_2 < \infty$. A function $f\in L^2(\mathcal{I})$ is called \emph{$\sigma$-uniform in $\mathcal{U}(\mathcal{I})$} if 
\begin{equation*}
\big|\int_{\mathcal{I}} f(x)\overline{u(x)}dx\big|\le \sigma \|f\|_{L^2(\mathcal{I})}, \text{ for every } u\in \mathcal{U}(\mathcal{I}). 
\end{equation*}
\end{defi}
The main result on $\sigma$-uniformity is the following.
\begin{lem}[\cite{Li}]\label{lemma:sigmaunif}
Let $\mathcal{L}$ be a bounded sublinear functional from $L^2(\mathcal{I})$ to $\C$, and $S_{\sigma}$ be the set of all functions that are $\sigma$-uniform in $\mathcal{U}(\mathcal{I})$. Denote 
\begin{equation*}A_{\sigma}=\sup\{|\mathcal{L}(f)|/\|f\|_{L^2(\mathcal{I})}: f\in S_{\sigma}, f\neq 0\},
\end{equation*}
and 
\begin{equation*}K=\sup_{u\in \mathcal{U}(\mathcal{I})}\frac{|\mathcal{L}(u)|}{\|u\|_2}.
\end{equation*}
Then 
\begin{equation*}\|\mathcal{L}\|\le \max\{A_{\sigma}, 2\sigma^{-1} K\}.
\end{equation*}
\end{lem}

In the following we will apply the lemma to the functional
\begin{equation}\label{eqn:sigmaunifdefL}
\mathcal{L}(g) = \int_\R \int_{\R} {f}(x +\lambda t)g(2^{-m_0}x+\lambda Q(t)) h(x) \tau(t) dtdx.
\end{equation}
Our goal is to prove \eqref{dualized-bound-linf}.
The interval $\mathcal{I}$ will mean either $[1/2,2]$ or $[-2,-1/2]$.

We also define 
\begin{equation*}\mathcal{U}(\mathcal{I})=\{ \eta\mapsto a(\xi, \eta) e^{i\alpha \eta+i \lambda\Psi(\xi,\eta)}: \alpha\in \R, 2^{-10}\le |\xi|\le 2^{10}\},
\end{equation*}
where $\lambda,\Psi$ are as defined in \eqref{eqn:lambdaQdef}, \eqref{psi-function} and $a(\xi,\eta)$ is a compactly supported smooth function that is to be determined later.

First we assume that $\widehat{g}|_{\mathcal{I}}$ is $\sigma$-uniform in $\mathcal{U}(\mathcal{I})$. 
Localizing in the spatial variable $x$, we write $\mathcal{L}(g)$ as
\begin{equation*}
\sum_{0\leq \iota < 2^m} \int_{\R}\int_{\R} (\mathbbm{1}_{J_\iota} f)(x+ \lambda t) g(2^{-m_0}x+\lambda Q(t))\tau(t) (\mathbbm{1}_{I_\iota} h)(x)dt dx,
\end{equation*}
where $I_\iota=2^{m_0}[\iota,\iota+1]$ and \begin{equation*}
J_\iota=[\iota 2^{m_0}+\frac12\lambda, (\iota+1)2^{m_0}+ 2\lambda]
\end{equation*}
We introduce this cumbersome notation because we need to keep track of the spatial localization of $f$ for technical reasons that become clear at the end of the argument. Denote  $f_\iota= \mathbbm{1}_{J_\iota} f$  and  $h_\iota=\mathbbm{1}_{I_\iota} h$.
Passing to the Fourier side, we obtain
\begin{equation}\label{eqn:sigmaunif1tobound}
\sum_{0\leq \iota < 2^m} \iiint_{\R^3} \widehat{f_\iota}(\xi) e^{ix\xi} \widehat{g}(\eta) e^{i2^{-m_0}x\eta} \Big( \int_{\R} \tau(t) e^{i \lambda  \Phi_{\xi,\eta}(t)} dt\Big) h_\iota(x) d(x,\xi,\eta),
\end{equation}
where $\Phi_{\xi,\eta}$ is as defined in \eqref{eqn:PhiDef}.
Due to the localization in $x$ we can replace $e^{i2^{-m_0} x\eta}$ by the constant $e^{i2^{-m_0}\alpha_\iota \eta}$, where $\alpha_\iota$ is an arbitrary point chosen from $I_\lambda$. More precisely, we write
\begin{equation*} e^{i2^{-m_0} x\eta} = e^{i2^{-m_0}\alpha_\iota \eta} \sum_{s=0}^\infty \frac{i^s}{s!} (2^{-m_0}(x-\alpha_\iota))^s \eta^s. 
\end{equation*}
Plugging this into \eqref{eqn:sigmaunif1tobound} we then proceed to treat every term of the Taylor expansion separately. However, since the treatment is the same for each of them we will here only show the argument for the term $s=0$ for simplicity of notation. Thus we are left with bounding
\begin{equation*}
\sum_{0\leq \iota < 2^m} \iiint_{\R^3} \widehat{f_\iota}(\xi) e^{ix\xi} \widehat{g}(\eta) e^{i2^{-m_0}\alpha_\iota\eta} \Big( \int_{\R} \tau(t) e^{i \lambda  \Phi_{\xi,\eta}(t)} dt\Big) h_\iota(x) d(x,\xi ,\eta). 
\end{equation*}
Appealing again to the stationary phase principle in the form of \eqref{eqn:statphase} leaves us with having to estimate
\begin{equation*}
\begin{split}
2^{-\frac{m}2}\sum_{0\leq \iota < 2^m} \int_{\R} \Big(\int_{\R} \widehat{g}(\eta)a(\xi,\eta) e^{i2^{-m_0}\alpha\eta+i \lambda\Psi(\xi,\eta)} d\eta\Big ) \widehat{f_\iota}(\xi) \widehat{h_\iota}(-\xi) d\xi,
\end{split}
\end{equation*}
where $a(\xi,\eta)$ is a compactly supported smooth function. Recall that $a(\xi,\eta)=0$ if there is no critical point and the remainder term from Lemma \ref{lemma:stationaryphase} is treated as in \eqref{eqn:statphaseerr}. 
Now we  apply the definition of $\sigma$-uniformity and Cauchy-Schwarz to bound the last expression by
\begin{equation}\label{eqn:sigmaunifcalc1}
2^{-\frac{m}2}\sigma\sum_{0\leq \iota< 2^m} \|f_\iota\|_2 \|h_\iota\|_2 \|g\|_2.
\end{equation}
Another application of the Cauchy-Schwarz inequality yields
\begin{equation*}\sum_{0\leq \iota< 2^m} \|f_\iota\|_2 \|h_\iota\|_2 \le \left( \sum_{0\leq \iota< 2^m}\|f_\iota\|_2^2 \right)^{1/2} \left( \sum_{0\leq \iota< 2^m} \|h_\iota\|_2^2 \right)^{1/2}.
\end{equation*}
This is where we need to make use of the spatial localization of $f_\iota$ on the interval $J_\iota$. We have
\begin{equation*} \sum_{0\leq \iota< 2^m}\|f_\iota\|_2^2 = \int_\R \left(\sum_{0\leq \iota< 2^m} \mathbbm{1}_{J_\iota}(x) \right) |f(x)|^2 dx \lesssim \max(1,2^{m-(d-1)j-\ell}) \|f\|_2^2. 
\end{equation*}
Thus, \eqref{eqn:sigmaunifcalc1} is bounded by
\begin{equation}\label{sigma-uniform-1}
 2^{-\frac{m}2} \sigma \max(1,2^{(m-(d-1)j-\ell)/2}) \|f\|_2 \|g\|_2 \|h\|_2\lesssim_{\ell,M} \sigma 2^{\frac{m_0}2} 2^{\frac{\kappa}2 m} \|f\|_2 \|g\|_2 \|h\|_\infty,
 \end{equation}
where in the last estimate we have used that $|m_0|>(1-\kappa)m$ and that $h$ is supported on $[0,2^{m+m_0}]$. This finishes the estimate for the case when $\widehat{g}|_{\mathcal{I}}$ is $\sigma$-uniform in $\mathcal{U}({\mathcal{I}})$. \\

It remains to treat the case when $\widehat{g}|_{\mathcal{I}}\in \mathcal{U}(\mathcal{I})$. We substitute $x\to 2^{m+m_0}x-\lambda t$ in in \eqref{eqn:sigmaunifdefL} and arrive at 
\begin{equation*}
2^{m+m_0} \int_{\R} f(2^{m+m_0} x)\int_{\R} g(2^{m} x-2^{-m_0}\lambda t+\lambda Q(t)) h(2^{m+m_0}x-\lambda t)\tau(t)dtdx
\end{equation*}
Applying H\"older's inequality  we bound this by $\|f\|_{2} \|T(g,h)\|_{2}$, where 
\begin{equation*}
\begin{split}
T(g,h)(x)&=2^{\frac{m+m_0}{2}} \int_{\R} g(2^{m} x-2^{-m_0}\lambda t+\lambda Q(t)) h(2^{m+m_0}x-\lambda t)\tau(t)dt.
\end{split}
\end{equation*}
Using the Fourier inversion formula we see that this is equal to\footnote{Up to a universal constant.}
\begin{equation*}
2^{\frac{m+m_0}{2}} \int_{\R}\int_{\R} \widehat{g}(\eta) e^{i\eta \lambda( 2^m\lambda^{-1} x-2^{-m_0}t+Q(t))} h(2^{m+m_0}x-\lambda t)\tau(t)d\xi dt.
\end{equation*}
Using our assumption on $\widehat{g}$ this becomes
\begin{equation*} 
2^{\frac{m+m_0}{2}} \int_{\R}\left(\int_{\R} a(\xi, \eta) e^{i\lambda (\eta z_{x,t}+\Psi(\xi,\eta))} d\eta \right) h(2^{m+m_0}x-\lambda t)\tau(t) dt.
\end{equation*}
Here we have set
\begin{equation}\label{eqn:zxtdef}
z_{x,t}=\alpha+2^{m}\lambda^{-1} x-2^{-m_0}t+Q(t),
\end{equation}
where $\alpha\in\R$ is arbitrary and $\xi$ is a parameter comparable to one whose precise value is irrelevant. We would like to apply the stationary phase principle. The phase function is
\begin{equation*}
\tilde{\Phi}_{x,t,\xi}(\eta)=\eta z_{x,t} + \Psi(\xi,\eta).
\end{equation*}
Similarly as above, we may assume that the equation 
\begin{equation*} 
z_{x,t} + \partial_\eta \Psi(\xi,\eta_c)=0.
\end{equation*}
has a unique solution $\eta_c=\eta_c(x,t,\xi)\in\mathcal{I}$. Recall that $\partial_\eta \Psi(\xi,\eta) = Q(t_c(\xi,\eta))$.
Let us write $t_c=t_c(\xi,\eta_c(x,t,\xi))$ in the following.
Set
\begin{equation*} \tilde{\Psi}_\xi(x,t) = \tilde{\Phi}_{x,t,\xi}(\eta_c) = \eta_c z_{x,t} + \Psi(\xi,\eta_c) 
\end{equation*}
Since
\begin{equation*} \Psi(\xi,\eta_c) = t_c \xi + \eta_c Q(t_c) 
\end{equation*}
we have
\begin{equation}\label{eqn:tildepsi}
\tilde{\Psi}_\xi(x,t) = t_c \xi = \xi Q^{-1}(-z_{x,t}).
\end{equation}
Calculate
\begin{equation*} |\partial_\eta^2 \Psi(\xi,\eta)| = \Big|\frac{(Q')^2}{\eta Q''}\Big| \approx_d 1. 
\end{equation*}
Thus, by stationary phase the major contribution to $T(g,h)(x)$ is
\begin{equation*}
2^{\frac{m_0}2} \int_{\R} \widetilde{a}_\xi(x,t) e^{i\lambda \tilde{\Psi}_\xi(x,t)} h(2^{m+m_0}x-\lambda t)\tau(t) dt,
\end{equation*}
for some compactly supported smooth function $\widetilde{a}_{\xi}$. Expanding the square of the $L^2$ norm of this expression and changing variables gives
\begin{equation*}
\begin{split}
2^{m_0} \int_{\R} \int_{\R} & \left(\int_{\R}  e^{i\lambda(\tilde{\Psi}_\xi(x,t)-\tilde{\Psi}_\xi(x,t+s))}\widetilde{a}_\xi(x,t)\overline{\widetilde{a}_\xi(x,t+s)}  \right.
 \\
& \times h(2^{m+m_0}x-\lambda t)\overline{h(2^{m+m_0}x-\lambda (t+s))}\tau(t)\overline{\tau(t+s)} dt \Big) dx
 ds.
\end{split}
\end{equation*}
The plan is to integrate by parts in the $t$ variable. Set
\begin{equation*} H(x) = h(2^{m+m_0}x)\overline{h}(2^{m+m_0}x-\lambda s),\quad \Theta(t) = \tau(t)\overline{\tau}(t+s) 
\end{equation*}
and $\chi_{\xi}(x,t)=\widetilde{a}_\xi(x,t)\overline{\widetilde{a}_\xi(x,t+s)}$.
Then our quantity equals
\begin{equation*}
2^{m_0}\int_\R\int_\R \Big(\int_\R e^{i\lambda(\tilde{\Psi}_\xi(x,t)-\tilde{\Psi}_\xi(x,t+s))} \Theta(t) H(x-2^{-m-m_0}\lambda t) dt\Big) \chi_{\xi}(x,t) dx ds.
\end{equation*}
Changing variables $x\mapsto x+2^{-m-m_0}\lambda t$ we get
\begin{equation}\label{0228e3.78}
\begin{split}
2^{m_0}\int_\R\int_\R \Big( \int_\R & e^{i\lambda (\tilde{\Psi}_\xi(x+2^{-m-m_0}\lambda t,t)-\tilde{\Psi}_\xi(x+2^{-m-m_0}\lambda t,t+s))} \Theta(t) b dt\Big) H(x) \\
& \times \chi_{\xi}(x+2^{-m-m_0}\lambda t,t) dx ds.
\end{split}
\end{equation}
From \eqref{eqn:zxtdef} and \eqref{eqn:tildepsi} we see that
\begin{equation*}
\tilde{\Psi}_\xi(x+2^{-m-m_0}\lambda t,t)-\tilde{\Psi}_\xi(x+2^{-m-m_0}\lambda t,t+s)
\end{equation*}
is equal to
\begin{equation}\label{0228e3.80}
\xi\cdot \left( Q^{-1} (- 
2^m\lambda^{-1} x - Q(t)) - Q^{-1} (- 2^m\lambda^{-1} x - 2^{-m_0}s - Q(t+s)) \right)
\end{equation}
Set 
\begin{equation*} \vartheta(x,t) =  Q^{-1} (- 
2^m\lambda^{-1} x - Q(t)) 
\end{equation*}
Writing $\vartheta=\vartheta(x,t)$ we have
\begin{equation*} \partial^2_t \vartheta(x,t) = - \frac{Q'(t)^2}{Q'(\vartheta)}\Big(\frac{Q''(t)}{Q'(t)^2}+\frac{Q''(\vartheta)}{Q'(\vartheta)^2 }\Big) 
\end{equation*}
We claim that 
\begin{equation}\label{0228e3.83}
|\partial^2_t \vartheta(x,t)| \gtrsim_{\ell} |x|.
\end{equation}
At this point we may assume without loss of generality that $Q$ is an odd polynomial. This is justified since it suffices to handle the cases that $Q$ is even or odd and the even case follows in the same way.
With that assumption we have
\begin{equation*}
| \partial^2_t \vartheta(x,t)| \approx \Big|\frac{Q''(t)}{Q'(t)^2}-\frac{Q''(Q^{-1}(Q(t)+\tilde{x}))}{Q'(Q^{-1}(Q(t)+\tilde{x}))^2 }\Big| \text{ with } \tilde{x}=2^m \lambda^{-1}x.
 \end{equation*}
Now \eqref{0228e3.83} follows immediately from the mean value theorem.

 By applying the mean value theorem again, together with the fact that $2^{-m_0} |s|$ is much smaller compared to $|s|$, we obtain that the $t$ derivative of \eqref{0228e3.80} is bounded from below by $|xs|$. Hence, by using the triangle inequality on small subsets around the origin in the variables $x$ and $s$ and integration by parts on the complement we obtain that
\begin{equation}\label{sigma-uniform-2}
|\eqref{0228e3.78} |\lesssim_\ell 2^{m_0} 2^{-\beta m} \|h\|^2_{\infty}
\end{equation}
for any $0<\beta<1$. Here we have used that both $x$ and $s$ take values in intervals of lengths which can be bounded by constants depending only on $d$ and $\ell$. 

Combining estimates \eqref{sigma-uniform-1} and \eqref{sigma-uniform-2} we obtain from Lemma \ref{lemma:sigmaunif} that 
\begin{equation*} 
|\mathcal{L}(g)|\lesssim_{\ell,M} 2^{\frac{m_0}2}\max\{\sigma 2^{\frac{\kappa}{2} m}, \sigma^{-1} 2^{-\frac{\beta}2 m}\} \|f\|_2 \|g\|_2 \|h\|_\infty.
\end{equation*}
Choosing $\kappa$ small enough, $\sigma=2^{-\kappa m}$ and $\beta=\frac12$, we can bound this by
\begin{equation*} \lesssim 2^{\frac{m_0}2-\frac{\kappa}2 m} \|f\|_2 \|g\|_2 \|h\|_\infty. 
\end{equation*}
Thus we have finished the proof of \eqref{dualized-bound-linf}.

\subsection{Case III: $d_0=1$.} By construction of the admissible sets $\mathcal{E}$ and $\Lambda$, there exists $d_1\not=1$ such that $-\ell\in \mathcal{J}^{(j)}_{1,d_1}$. That is, the linear term dominates at the dyadic scale $-\ell$ and the $d_1$th power monomial dominates the remaining, nonlinear monomials at that scale. We have
\begin{equation}\label{eqn:d1dominating}
 |a_{d_1} (2^{j-\ell})^{d_1}|> \Gamma_0|a_r (2^{j-\ell})^r| \text{ for every } r\not\in\{1, d_1\}.
\end{equation} 
Note that $m_0 = b_1$ and let us write
\begin{equation*}
Q(t)=2^{-b_1} a_1 t+R(t).\end{equation*}
For convenience let us assume that $a_1=2^{b_1}$. This does not affect the argument we give, but simplifies notation.
From \eqref{eqn:d1dominating} and \eqref{eqn:lambdaQdef} we see that $\|R\|_{C^d([1/2,2])} \approx 2^{q_0}$, where
\begin{equation}\label{eqn:q0est}
0>q_0 := b_{d_1} + (d_1-1)(j-\ell) - b_1>(d-1)(j-\ell)-b_1.
\end{equation}
Also, since the linear term is dominating,
\[
(j-\ell)(d-1)< b_1 \le \log(M).
\]
Since $j\ge 0$ we therefore have that $|b_1|$ and $|q_0|$ are comparable to one up to constants depending on $d,\ell, M$. 
Applying the stationary phase principle in the form of Lemma \ref{lemma:stationaryphase} and discarding the remainder term based on the same argument that led to \eqref{eqn:statphaseerr}, it now remains to prove
\begin{equation}\label{181118e4.36}
\left|\int_{\R}\int_{\R} \widehat{f}(\xi) \widehat{g}(\eta) e^{i \lambda\Psi(\xi, \eta)} a(\xi,\eta) \widehat{h}(-\xi-2^{-b_1}\eta) d\xi d\eta\right|  \lesssim 2^{-\gamma m}\|f\|_2 \|g\|_2 \|h\|_2
\end{equation}
for some positive $\gamma>0$ and $\widehat{f},\widehat{g}$ supported in the annulus $1\le |\xi|\le 2$. Here $a$ is a smooth and compactly supported cutoff function and we keep in mind that in the support of $a$ we have $|\xi+\eta| \approx 2^{q_0},$ since this is a necessary condition for the existence of a stationary point.
\begin{claim}\label{claim4.5}
There exist $\gamma>0$ and an integer $C_d$ depending only on the degree $d$ and intervals $W_1, \dots, W_{C_d}\subset \R$ of length at most $2^{-\gamma m}$, such that whenever $\xi/\eta\not\in W_{\iota}$ for all $\iota$, we have 
\begin{equation}\label{181118e4.37}
|\partial_{\xi}\partial_{\eta}(2^{-b_1}\partial_{\xi}\Psi-\partial_{\eta}\Psi)|\gtrsim 2^{-\gamma m}.
\end{equation}
Implicit constants depend on $d,\ell$ and $M$. 
\end{claim}

This claim will be proven at the end of this section. A condition of the form \eqref{181118e4.37} first appeared in the work of Li \cite{Li}, see also Xiao \cite{Xi17} and Gressman and Xiao \cite{GX16}.  Let $\widetilde{a}: \R\to \R$ be a smooth bump function that is equal to one on each enlarged interval $2 W_{\iota}$ such that $\|\widetilde{a}\|_{C^4}\lesssim 2^{4\gamma m}$. We bound the left hand side of \eqref{181118e4.36} by the sum of 
\begin{equation}\label{181118e4.38}
\int_{\R}\int_{\R} |\widehat{f}(\xi) \widehat{g}(\eta) a(\xi,\eta)\widetilde{a}(\xi/\eta) \widehat{h}(-\xi-2^{-b_1}\eta)| d\xi d\eta
\end{equation}
and 
\begin{equation}\label{181118e4.39}
\left|\int_{\R}\int_{\R} \widehat{f}(\xi) \widehat{g}(\eta) e^{i \lambda\Psi(\xi, \eta)} a(\xi,\eta)(1-\widetilde{a}(\xi/\eta)) \widehat{h}(-\xi-2^{-b_1}\eta) d\xi d\eta\right|.
\end{equation}
By the Cauchy-Schwarz inequality we obtain 
\[\eqref{181118e4.38}\lesssim_{l, M} 2^{-\frac{\gamma m}{2}} \|f\|_2\|g\|_2\|h\|_2.
\]

It remains to control \eqref{181118e4.39}. Here we follow a similar argument as in Case I.
Write $b(\xi,\eta)=a(\xi,\eta)(1-\widetilde{a}(\xi/\eta))$. Then $\|b\|_{C^4}\lesssim 2^{4\gamma m}$. After applying a change of variables and the Cauchy-Schwarz inequality, we conclude that it is enough to prove 
\begin{equation}\label{181118e4.40}
\Big\|\int_{\R} \widehat{f}(\xi-2^{-b_1}\eta) \widehat{g}(\eta)b(\xi-2^{-b_1}\eta, \eta) e^{i\lambda \Psi(\xi-2^{-b_1}\eta, \eta)} d\eta \Big\|^2_{L^2_{\xi}} \lesssim 2^{-\gamma m} \|f\|_2^2\|g\|_2^2.
\end{equation}
By the triangle inequality, it suffices to prove \eqref{181118e4.40} with a better gain $2^{-6\gamma m}$ in place of $2^{-\gamma m}$ for every function $g$ with $\widehat{g}$ supported on an interval of length $2^{-2\gamma m}$. We expand the square on the left hand side of \eqref{181118e4.40}. After a change of variable, we obtain 
\begin{equation}\label{181118e4.41}
\int_{|\alpha|\le 2^{-2\gamma m}} \iint_{\R^2} e^{i\lambda [\Psi(\xi, \eta)-\Psi(\xi+2^{-b_1}\alpha, \eta-\alpha)]} F_{\alpha}(\xi)G_{\alpha}(\eta) a_{\alpha}(\xi, \eta)d\xi d\eta d\alpha,
\end{equation}
where $F_{\alpha}, G_{\alpha}$ are as in \eqref{eqn:TTstarnotation}, and $a_{\alpha}$ is some new compactly supported amplitude. By the mean value theorem we see that
\begin{equation}
\Big|\partial_{\xi}\partial_{\eta} \Big(\Psi(\xi, \eta)-\Psi(\xi+2^{-b_1}\alpha, \eta-\alpha) \Big) \Big| \gtrsim 2^{-\gamma' m}|\alpha|. 
\end{equation}
Now we are ready to apply Lemma \ref{lemma_hormander} to bound \eqref{181118e4.41} in the same way as we did in \eqref{mainest:alphabig}. This concludes the proof of \eqref{181118e4.40}.

\begin{proof}[Proof of Claim \ref{claim4.5}.]
Recall that $t_c(\xi, \eta)$ is defined via 
\begin{equation}
\xi+ \eta +\eta R'(t_c)=0.
\end{equation}
We write $\rho = R'(t_c) = - \tfrac{\xi+\eta}{\eta}$. Recall that
\begin{equation}
\Psi(\xi, \eta)= t_c\cdot (\xi+\eta)+ \eta R(t_c).
\end{equation}
Direct computation shows that the expression $2^{-b_1}\cdot\partial_\xi^2 \partial_\eta \Psi(\xi,\eta) - \partial_{\xi}\partial_\eta^2 \Psi(\xi,\eta)$ equals
\[ \frac{\eta (2\xi+2^{-b_1}\eta) (R''(t_c))^2 + \xi (\xi+2^{-b_1}\eta) R'''(t_c)}{\eta^4 R''(t_c)^3}. \]
We have
\[ \eta(2\xi+2^{-b_1}\eta) = -2\eta^2 (\rho + (1-\tfrac12 2^{-b_1})),\;\text{and} \]
\[ \xi(\xi+B\eta) = \eta^2 (\rho+1)(\rho+(1-2^{-b_1})). \]
Therefore, the left hand side of the claimed inequality is comparable (up to constants depending on $d,\ell,M$) to the absolute value of
\begin{equation}\label{eqn:lineartermpfphase1}
-2(\rho+(1-2^{-b_1-1}))(R''(t_c))^2 + (\rho+1)(\rho+1-2^{-b_1}) R'''(t_c),
\end{equation}
which equals
\[
-2(R'(t_c)+(1-2^{-b_1-1}))(R''(t_c))^2 + (R'(t_c)+1)(R'(t_c)+1-2^{-b_1}) R'''(t_c),
\]
As a function of $t_c$, this is a polynomial of degree $3d-5$ with bounded coefficients (by a constant depending only on $d,\ell,M$) and leading coefficient away from zero. This implies the claim.
\end{proof}

\end{document}